\documentclass{amsart}

\usepackage{amsmath}
\usepackage{amssymb}
\usepackage{hyperref}
\usepackage{cases}
\usepackage{cleveref}
\usepackage{xcolor}

\newtheorem{theorem}{Theorem}[section]
\newtheorem{prop}[theorem]{Proposition}
\newtheorem{lemma}[theorem]{Lemma}
\newtheorem{cor}[theorem]{Corollary}

\theoremstyle{definition}
\newtheorem{definition}[theorem]{Definition}

\theoremstyle{remark}

\def\ep{\varepsilon}    

\newcommand{\R}{\mathbb{R}}

\begin{document}

\title{Smoothness of conformal heat flow of harmonic maps}

\author{Woongbae Park}
\address{803 Hylan Building, Department of Mathematics, University of Rochester, Rochester, New York, 14627}
\email{wpark14@ur.rochester.edu}



\begin{abstract}
The conformal heat flow of harmonic maps is a system of evolution equations combined with harmonic map flow with metric evolution in conformal direction.
It is known that global weak solution of the flow exists and smooth except at mostly finitely many singular points.
In this paper, we show that no finite time singularity occurs, unlike the usual harmonic map flow.
And if the initial energy is small, we can obtain the uniform convergence of the map to a point and the conformal factor of the metric under some time sequence $t_n \to \infty$.
Also, under the assumption that energy concentration is uniform in time, we show that there exists a sequence of time $t_n \to \infty$ such that $f(\cdot,t_n)$ converges to a harmonic map in $W^{1,2}$ on any compact set away from at most finitely many points.
\end{abstract}

\maketitle
\sloppy

\section{Introduction}
\label{sec1}

We consider the following conformal heat flow (CHF) for harmonic maps. It is a solution $(f,u)$ of a pair of equations
\begin{equation}\label{eq1}
\begin{cases}
f_t =& \tau_g (f)\\
u_t =& b |df|_g^2 - a
\end{cases}
\end{equation}
where $a,b$ are positive constants.
Here $f : (M,g) \times [0,T] \to (N,h)$ is a time-dependent map from a Riemann surface $(M,g)$ without boundary to a compact Riemannian manifold $(N,h)$, $u : (M,g) \times [0,T] \to \R$ is a conformal factor of the metric $g$ given by $g(x,t) = e^{2u(x,t)}g_0(x)$, $\tau_g(f)$ is the tension field of $f$ with respect to $g$ and $|df|_g^2 = g^{ij} h_{\alpha \beta} f^{\alpha}_i f^\beta_j$ is the energy density in local coordinates.
We put the initial condition $f(0) = f_0 : (M,g) \to (N,h)$ and $u(0) = 0$.
If we embed $(N,h) \hookrightarrow \R^L$ isometrically, the equation can be written with respect to $g_0$ by
\begin{equation} \label{eq2}
\begin{cases}
f_t =& e^{-2u} \tau_{g_0}(f) = e^{-2u} (\Delta f + A(f)(df,df))\\
u_t =& b e^{-2u} |df|^2 - a.
\end{cases}
\end{equation}
We assume that there is a constant $C_N$ which only depends on the embedding $(N,h) \hookrightarrow \R^L$ such that $\|A\|_{\infty}, \|DA\|_{\infty}, \|D^2 A\|_{\infty} \leq C_N$ and that $b > 2C_N^2 + C_N$.

The motivation of CHF for harmonic maps is introduced in \cite{P23}.
It is designed to obtain a variation of harmonic map flow which has a smooth global solution.
For a map $f : (M,g) \to (N,h)$ between two Riemannian manifolds, the harmonic map flow is a negative gradient flow of Dirichlet energy
\[
E(f) = \frac{1}{2} \int_{M} |df|_g^2 dvol_{g}
\]
where $dvol_{g} = \sqrt{|g|} dx$ is the volume form of $M$.
It is known that the harmonic map heat flow exists and smooth under non-positive sectional curvature condition of $N$ by Eells-Sampson \cite{EL78}, and global weak solution was obtained by Struwe \cite{S85} where at most finitely many bubble points can occur.
And such finite time singularity indeed exists, by Chang-Ding-Ye \cite{CDY92} and recently generalized by Davila-Del Pino-Wei \cite{DDW20}.
The further results, we refer to, e.g., Qing-Tian \cite{QT97}, Chang \cite{Chang89}, Topping \cite{T97}, \cite{T02}, \cite{T04b}, Freire \cite{F95}, Ling-Wang \cite{LW98}, and many others.

In terms of variations of harmonic map flow together with evolution of the metric, Topping-Rupflin \cite{RT18a} developed the so-called Teichm{\"u}ller flow which is the negative $L^2$ gradient flow of the Dirichlet energy with respect to both the map and the metric.
M{\"u}ller \cite{M12} considered Ricci-harmonic map flow, a combination of harmonic map flow and Ricci flow of the metric, and Buzano-Rupflin \cite{BR17} showed the long time existence under large coupling constant for 2 dimensional domain.

In a previous paper \cite{P23}, the author studied the CHF of harmonic maps \eqref{eq1} and established the following analogue of Struwe's theorem.
\begin{theorem}[Park \cite{P23}] \label{weak thm}
There exists a weak solution $f \in W^{1,2}_{loc}(M \times [0,\infty),N)$, $u \in W^{1,2}_{loc}(M \times [0,\infty))$ of \eqref{eq1} that is smooth in $M \times [0,\infty)$ except at most finitely many points.
\end{theorem}

This paper improves \Cref{weak thm} by studying the possible finite-time singularities, and showing that, in fact, no such singularities exist.

\begin{theorem} \label{main1}
Any solution $(f,u)$ of \eqref{eq1} is smooth on $M \times [0,\infty)$.
\end{theorem}

To show this, we assume there is a finite time singularity and get a contradiction.
By \cite{P23}, the only singularity type is due to the energy concentration, unlike the Teichm{\"u}ller flow where metric degeneration may occur.
For a smooth solution $(f,u)$ of \eqref{eq1} on $M \times [0,T)$, we say that the flow blows up at $(x_0,T)$ if
\begin{equation} \label{blow up}
\lim_{r \to 0} \lim_{t \nearrow T} \frac{1}{2} \int_{B_r(x_0)} |df|^2(t) \geq \ep_1
\end{equation}
for some $\ep_1>0$.
These singular points are called bubble points.

One can try to build bubbles by usual method which requires sequence of time $t_n \nearrow T$ and $L^2$ control of tension field $\|\tau(f)(t_n)\|_{L^2} \to 0$, see Qing \cite{Q95}.
However, for all $t_1 < t_2 < T$, we only have weaker control
\begin{equation} \label{tension control}
\iint_{M \times [t_1,t_2]} e^{-2u} |\tau(f)|^2 = E(t_1) - E(t_2)
\end{equation}
and $e^{-2u}(x,t) \to 0$ if $(x,t)$ approaches singularity.
At this point, we even do not have $\|\tau(f)(t)\|_{L^2} \leq C$.
Thanks to the estimate in \cite{P23}, however, we can overcome this difficulty and obtain that no finite time bubbling occurs.

For infinite time bubbling, the same issue exists.
In the case of usual harmonic map flow, one can have $\|\tau(f)(t_n)\|_{L^2} \to 0$ for some time sequence $t_n \to \infty$ and using it to get bubbles at infinite time and show the energy identity.
Again, this is not the case for CHF.
For more general control of the tension field, Qing-Tian \cite{QT97} and Ding-Tian \cite{DT95} showed energy identity and no-neck property for $\|\tau(f_i)\|_{L^2} < C$. Wang \cite{Wang17} considered the case of $\|\tau(f_i)\|_{M^{1,\delta}} < C$ where $M^{1,\delta}$ is the Morrey space for some $0 \leq \delta < 2$, and Wang-Wei-Zhang \cite{WWZ17} showed the same results for $\|\tau(f_i)\|_{L^p} < C$ for some $p>1$.
Note that Parker \cite{P96} showed that $\|\tau(f_i)\|_{L^1} < C$ is not enough for a no-neck property.

On the other hand, for CHF, we cannot choose a sequence of time $t_n \to \infty$ such that $f(t_n)$ be a sequence of approximate harmonic maps with $L^2$ bounded tension fields due to the weaker control of tension field \eqref{tension control}.
But if the initial energy is small enough, one can recover convergence over a sequence of time $t_n \to \infty$ similar to usual harmonic map flow, see Wang \cite{Wang99} or Lin \cite{Lin03}.

\begin{theorem} \label{main2}
There exists $\ep_2>0$ such that if $(f,u)$ be a solution of \eqref{eq1} with $E(f(0)) \leq \ep_2$, then 
$\|df\|_{L^\infty(M)}(t) \to 0$ as $t \to \infty$ and $u(x,t) \to -\infty$ as $t \to \infty$ for all $x \in M$.

In particular, there exists $t_n \to \infty$ such that $f(\cdot,t_n)$ converges to a point uniformly.
\end{theorem}

Without small energy assumption, we can have the following convergence result at time infinity away from finitely many bubble points.

\begin{theorem} \label{main3}
Let $(f,u)$ be a solution of \eqref{eq1} and assume that all sequential bubble points are uniform bubble points.
Then there exists $t_n \to \infty$ such that $f(\cdot,t_n)$ converges to a harmonic map $f_{\infty} (\cdot)$ in $W^{1,2}$ on any compact set $K \subset M \setminus S$ where $S = \{x_1, \cdots, x_k\}$ is the set of uniform bubble points.
\end{theorem}

For the definition of sequential bubble points and uniform bubble points, see \Cref{bubble points}.
The assumption that all sequential bubble points are uniform bubble points says that the energy concentration is uniform in time.
This assumption is needed to avoid the case when a bubble is traveling inside $N$ forever and does not converge.

Throughout the paper, all the computations are obtained with respect to the metric $g_0$ unless the metric is specified.

\section{Local and global estimate}
\label{sec2}

In this section we show some estimates about CHF for harmonic maps.
First we review that how local energy behaves and when we get smoothness locally.
Then we develop parallel global version of the estimates similar to \cite{P23}.

\begin{lemma}[Park \cite{P23}, Lemma 2] \label{volume}
Let $(f,u)$ be a solution of \eqref{eq2} on $M \times [0,\infty)$.
Then the volume with respect to the metric $g(t) = e^{2u}g_0$ is finite for all time.
Also, for any $t_1 < t_2$,
\begin{equation}
E(t_1) - E(t_2) = \int_{t_1}^{t_2} \int_{M} e^{2u}|f_t|^2.
\end{equation}
\end{lemma}

Note that the energy over whole manifold $M$ is decreasing.
Hence we denote $E(t) \leq E(0) = E_0$.

The following lemma is so-called energy gap property of harmonic maps.

\begin{lemma}[Parker \cite{P96} or Sacks-Uhlenbeck \cite{SU81}] \label{energy gap}
There is $\ep_0>0$ such that, if $f : M \to N$ is a nontrivial harmonic maps, then $E(f) \geq \ep_0$.
\end{lemma}

Even though global energy is decreasing, local energy may be increasing at some regions.
Similar to \cite{S85}, we have the following local energy lemma, saying that local energy can increase by the amount controlled by $\frac{t_2-t_1}{(r_2-r_1)^2}$.

\begin{lemma}[Park \cite{P23}, Proposition 14] \label{local energy lemma}
Let $(f,u)$ be a solution of \eqref{eq2} on $M \times [0,T]$ and smooth on $M \times [0,T)$.
Choose $t_1 < t_2 < T$ and some ball $B_{2r}$ with cut-off function $\varphi \in C^\infty_c(B_{2r})$ such that $|\nabla \varphi| \leq \frac{2}{r}$.
Then
\begin{equation}
\begin{aligned}
\int_{t_1}^{t_2} \int_{B_{2r}} e^{2u}|f_t|^2 \varphi^2 +& \int_{B_{2r}}|df|^2 \varphi^2 (t_2) - \int_{B_{2r}}|df|^2 \varphi^2 (t_1)\\
&\le \frac{4^2}{ar^2} (e^{2at_2}-e^{2at_1}) E_0.
\end{aligned}
\end{equation}
In particular, for $r_1<r_2$ and $t_1 < t_2<T$,
\begin{equation} \label{local energy est}
E(B_{r_1},t_2) - E(B_{r_2},t_1) \leq C \frac{t_2-t_1}{(r_2-r_1)^2}
\end{equation}
where $C$ depends on $E_0,T$.
\end{lemma}

Next, we investigate when the finite time singularity occurs.
The author in \cite{P23} showed that if local energy is small and conformal factor is small, we have local smoothness.

\begin{prop}[Park \cite{P23}, Lemma 25] \label{ep reg}
There is $\ep_1>0$ such that if $(f,u)$ be a smooth solution of \eqref{eq2} on $B_{2r} \times [T-\delta r^2,T)$ satisfying, for all $r$ small enough,
\[
\sup_{t \in [T-\delta r^2,T]} E(B_{2r},t) \leq \ep_1 \quad \text{ and } \quad \int_{B_{2r} \times \{T-\delta r^2\}} e^{18u} \leq \ep_1,
\]
then H{\"o}lder norms of $f$ and $u$ on $B_r \times [T-\delta r^2,T]$ and their derivatives are all bounded by constants only depending on $T$, $\delta$, $r$, $\ep_1$ and $C_N$.
\end{prop}

To achieve the assumptions, we first fix $t_1$ and let $\delta = \delta(t_1,r,T)$ such that $T-\delta r^2 = t_1$.
Then the condition $\int_{B_{2r} \times \{t_1\}} e^{18u} \leq \ep_1$ can be achieved if we choose $r$ small enough, because the integral is only computed at initial time $t_1$ when $u$ is assumed to be smooth.
Finally, if $\lim_{r \to 0} \lim_{t \nearrow T} E(B_{2r},t) \leq \ep_1$, then we choose $r$ small enough to achieve the first assumption as well.
Hence, the only possibility for finite time singularity is when \eqref{blow up} holds.
This means that any singularity of CHF is not due to the geometric evolution but only due to the energy concentration as usual harmonic map flow.
So this proposition validates blow up condition \eqref{blow up}.

We state several important lemmas that are similar to those in \cite{P23} but over the whole domain manifold $M$.

\begin{lemma}
Let $(f,u)$ be a smooth solution of \eqref{eq2} on $M \times [0,T)$.
For $p \geq 0$, we have
\begin{equation} \label{f_t p+2 der}
\begin{split}
\frac{d}{dt} \int_{M} e^{2u}|f_t|^{p+2} \leq& \,\, 2a(p+1) \int_{M} e^{2u}|f_t|^{p+2}  - \frac{p+2}{2} \int_{M} |\nabla f_t|^2 |f_t|^p\\
& + \left( (p+2)C_N + 2(p+2)C_N^2 - 2b(p+1) \right) \int_{M} |df|^2 |f_t|^{p+2}.
\end{split}
\end{equation}
\end{lemma}

\begin{proof}
Taking time-derivative of the first equation in \eqref{eq2} to get
\[
(e^{2u}f_t)_t = \Delta f_t + \frac{\partial}{\partial t}A(df,df).
\]
Taking inner product with $f_t |f_t|^p$ and integrating over $M$ gives
\[
\begin{split}
\int_M \langle (e^{2u}f_t)_t, f_t |f_t|^p \rangle =& \int_{M} \langle \Delta f_t, f_t |f_t|^p \rangle + \int_{M} \langle \frac{\partial}{\partial t}A(df,df),f_t |f_t|^p \rangle\\
=& -\int_{M} |\nabla f_t|^2 |f_t|^p - p \int_{M} (\langle \nabla f_t, f_t \rangle)^2 |f_t|^{p-2}\\
&+ \int_{M} \langle DA(df,df) \cdot f_t, f_t \rangle |t_t|^p + 2\int_{M} \langle DA(df,df_t),f_t |f_t|^p \rangle\\
\leq& -\int_{M} |\nabla f_t|^2 |f_t|^p + C_N \int_{M} |df|^2 |f_t|^{p+2} + 2C_N \int_{M} |df| |\nabla f_t| |f_t|^{p+1}\\
\leq& -\frac{1}{2} \int_{M} |\nabla f_t|^2 |f_t|^p + (C_N + 2C_N^2) \int_{M} |df|^2 |f_t|^{p+2}.
\end{split}
\]
On the other hand, LHS becomes
\[
\begin{split}
\int_{M} \langle (e^{2u}f_t)_t, f_t |f_t|^p \rangle =& \frac{1}{p+2} \frac{d}{dt} \int_{M} e^{2u}|f_t|^{p+2} + 2 \frac{p+1}{p+2} \int_{M} e^{2u} |f_t|^{p+2} u_t\\
=& \frac{1}{p+2} \frac{d}{dt} \int_{M} e^{2u} |f_t|^{p+2} + 2b \frac{p+1}{p+2} \int_{M} |df|^2 |f_t|^{p+2}\\
& - 2a \frac{p+1}{p+2} \int_{M} e^{2u} |f_t|^{p+2}.
\end{split}
\]
Combining the two, we get
\[
\begin{split}
\frac{d}{dt} \int_{M} e^{2u}|f_t|^{p+2} \leq& 2a(p+1) \int_{M} e^{2u}|f_t|^{p+2} - \frac{p+2}{2} \int_{M} |\nabla f_t|^2 |f_t|^p\\
&+ \left( (p+2)C_N + 2(p+2)C_N^2 - 2b(p+1) \right) \int_{M} |df|^2 |f_t|^{p+2}.
\end{split}
\]
\end{proof}

Now from the condition of $b > 2C_N^2 + C_N$,
\begin{equation}
-C_4(p) := (p+2)C_N + 2(p+2)C_N^2 - 2b(p+1) <0
\end{equation}
for all $p \geq 0$.
So,
\[
\frac{d}{dt} \int_{M} e^{2u}|f_t|^{p+2} \leq 2a(p+1) \int_{M} e^{2u}|f_t|^{p+2}
\]
and for any $t_1<t_2$, we get
\begin{equation} \label{f_t p+2}
\int_{M} e^{2u}|f_t|^{p+2} (t_2) - \int_{M} e^{2u}|f_t|^{p+2} (t_1) \leq 2a(p+1) \int_{t_1}^{t_2} \int_{M} e^{2u}|f_t|^{p+2}.
\end{equation}

In particular, for any $t>0$ we have
\begin{equation} \label{e 2u f_t est}
\int_{M} e^{2u}|f_t|^2 (t) \leq \int_{M} e^{2u}|f_t|^2 (0) + 2a \int_{0}^{t} \int_{M} e^{2u}|f_t|^2 \leq K(0) + 2a E(0)
\end{equation}
where we denote $K(0) = \int_{M} e^{2u}|f_t|^2 (0) $.

Integrating \eqref{f_t p+2 der} from $t_1$ to $t_2$ gives
\begin{equation} \label{f_t p+2 full}
\begin{split}
C_4(p) \int_{t_1}^{t_2} \int_{M} |df|^2 |f_t|^{p+2} &+ \frac{p+2}{2} \int_{t_1}^{t_2} \int_{M} |\nabla f_t|^2 |f_t|^p + \int_{M} e^{2u}|f_t|^{p+2}(t_2) - \int_{M} e^{2u}|f_t|^{p+2}(t_1)\\
\leq& \,\, 2a(p+1) \int_{t_1}^{t_2} \int_{M} e^{2u}|f_t|^{p+2}.
\end{split}
\end{equation}

\begin{prop} \label{int df^2 f_t^2}
Let $(f,u)$ be a smooth solution of \eqref{eq2} on $M \times [0,T)$.
Then for any $0 \leq t_1<t_2 \leq T$, we have
\begin{equation}
\int_{t_1}^{t_2}\int_{M} |df|^2 |f_t|^2 \leq C
\end{equation}
for some constant $C$ only depending on $E(0),K(0),T,C_4(0)$.
\end{prop}

\begin{proof}
From \eqref{f_t p+2 full} and \eqref{e 2u f_t est}, we get the conclusion.
\end{proof}

\section{Proof of \Cref{main1}}
\label{sec3}

Now we prove the first main theorem.
We assume there is a finite time blow up at $(x_0,T)$.
Since blow up is a local nature, we only focus on some small ball $B$ centered at the bubble point $x_0$.
Now let $(f,u)$ be a smooth solution of CHF on $B \times [0,T)$ which blows up at $(x_0,T)$.
The blow up assumption implies
\begin{equation} \label{energy amount}
\lim_{t \nearrow T} \frac{1}{2}\int_{B}|df|^2(t) = K + \frac{1}{2}\int_{B} |df|^2 (T)
\end{equation}
where $K>0$ is the energy loss at $(x_0,T)$.
Equivalently, $K$ can be obtained by
\[
K = \lim_{r \to 0} \lim_{t \nearrow T} \frac{1}{2}\int_{B_r(x_0)} |df|^2(t).
\]
From \eqref{blow up}, $K \geq \ep_1$.

Let $\varphi \in C^{\infty}(B)$ be a cut-off function supported in $B_r = B_r(x_0)$ such that $0 \leq \varphi \leq 1$ on $B_r$, $\varphi \equiv 1$ on $B_{r/2}$ and $|\nabla \varphi| \leq \frac{C}{r}$.
For each $t \in [0,T)$, as in \cite{T04b}, define the local energy
\begin{equation}
\Theta_r(t) = \frac{1}{2} \int_{B_r} \varphi^2 |df(t)|^2(y) dy.
\end{equation}
Then
\[
\begin{split}
\frac{d \Theta}{dt} =& \int_{B_r} \varphi^2 \langle df, df_t \rangle = - \int_{B_r} \varphi^2 \langle \Delta f, f_t \rangle - 2 \int_{B_r} \varphi \langle df, \nabla \varphi \cdot f_t \rangle\\
=& - \int_{B_r} \varphi^2 e^{2u} |f_t|^2 - 2 \int_{B_r} \varphi \langle df, \nabla \varphi \cdot f_t \rangle\\
\leq& \frac{C}{r} \int_{B_r}  |df| |f_t| \\
 \leq& C \left( \int_{B_r} |df|^2 |f_t|^2 \right)^{\frac{1}{2}}.
\end{split}
\]
So
\begin{equation} \label{Theta eq}
\begin{split}
\Theta_r(t) - \Theta_r(s) \leq& C \int_{s}^{t} \left(  \int_{B_r}  |df|^2 |f_t|^2 \right)^{\frac{1}{2}} \\
\leq& C (t-s)^{\frac{1}{2}} \left( \int_{s}^{t} \int_{B_r} |df|^2 |f_t|^2 \right)^{\frac{1}{2}} \\
\leq& C (t-s)^{\frac{1}{2}}
\end{split}
\end{equation}
by \Cref{int df^2 f_t^2}.
Hence $\Theta_r(t)$ has a limit as $t \nearrow T$ and we have
\[
K = \lim_{t \nearrow T} \frac{1}{2}\int_{B}|df|^2(t) - \frac{1}{2}\int_{B} |df|^2 (T) = \lim_{t \nearrow T} \Theta_r(t) - \Theta_r(T)
\]
for any $r$.
So, take $t \nearrow T$ in \eqref{Theta eq} to get
\begin{equation} \label{K est}
\left| K + \Theta_{r}(T) - \Theta_{r}(s) \right| \leq C (T-s)^{\frac{1}{2}}
\end{equation}
for any $r$ and $s<T$.

Now fix $\delta>0$ such that
\begin{equation}
\delta < \max \{ \frac{1}{C}, \frac{K}{2}\}
\end{equation}
where $C,K$ are from \eqref{K est}.
Fix $s = T-\delta^4$ and choose $\ep$ small enough such that $|\Theta_{\ep r}(T)| < \frac{K}{10}$ and $|\Theta_{\ep r}(s)| < \frac{K}{10}$.
Then we can replace $r$ by $\ep r$ in \eqref{K est} to get
\[
\frac{8K}{10} < C (T-s)^{\frac{1}{2}} = C \delta^2 < \delta < \frac{K}{2}
\]
which is a contradiction.
Hence there is no finite time bubbling and the \Cref{main1} is proved.

\section{Infinite time analysis}
\label{sec4}

Harmonic map flows may develop finite time bubbling or infinite time bubbling.
And for infinite time bubbling, the limit map becomes harmonic and there is a sequence of time $t_n \to \infty$ such that the tension fields at $t_n$ tends to zero in $L^2$.
Hence, unlike finite time bubbling, no-neck properties can be obtained in infinite time bubbling.
(Note that energy identity can be obtained in both cases.)
This observation leads to the finite time bubbling example, if the initial map is in the homotopy class that contains harmonic spheres far from each other.

Now if we run the CHF for the same initial condition, since no finite time bubbling occurs, there should be infinite time bubbling.
But we even do not know whether the limit map is harmonic.
In fact, we can find sequences $t_n \to \infty$ such that
\[
t_n \int_{M} e^{-2u} |\tau(f)|^2(t_n) \to 0
\]
which is not strong enough so that we cannot guarantee $\|\tau(f)(t_n) \|_{L^2} \to 0$.

In this section, we try to overcome this difficulties by using extra regularity of CHF.

We first need the following higher estimate.

\begin{lemma}\label{higher est}
Let $(f,u)$ be a smooth solution of \eqref{eq2} on $M \times [0,\infty)$.
For $p \geq 1$ and $t_1 < t_2$, we have
\begin{equation}
\begin{split}
\iint_{M \times [t_1,t_2]} e^{2u}|f_t|^{p+2} \leq C  \iint_{M \times [t_1,t_2]} |\nabla f_t|^2 |f_t|^{p-1} + C \iint_{M \times [t_1,t_2]} |df|^2 |f_t|^{p+1}
\end{split}
\end{equation}
where $C$ only depends on $p$.
\end{lemma}

\begin{proof}
From the first equation in \eqref{eq2}, taking inner product with $f_t |f_t|^{p}$ and integrating over $M \times [t_1,t_2]$ gives
\[
\begin{split}
\iint_{M \times [t_1,t_2] } e^{2u}|f_t|^{p+2} =& \iint_{M \times [t_1,t_2]} \langle f_t, \Delta f \rangle |f_t|^p\\
=& -\iint_{M \times [t_1,t_2]} \langle \nabla f_t, \nabla f \rangle |f_t|^p\\
& - p \iint_{M \times [t_1,t_2]} \langle f_t, \nabla f \rangle |f_t|^{p-2} \langle f_t, \nabla f_t \rangle\\
\leq& C \iint_{M \times [t_1,t_2]} |\nabla f_t|^2 |f_t|^{p-1} + C \iint_{M \times [t_1,t_2]} |df|^2 |f_t|^{p+1}.
\end{split}
\]
\end{proof}

Using above lemma and step-by-step increasing of regularity, we can get the following proposition.

\begin{prop} \label{t_n p}
Let $(f,u)$ be a smooth solution of \eqref{eq2} on $M \times [0,\infty)$.
For any $p \geq 0$, we have
\begin{align}
\int_{0}^{\infty} \int_{M} |df|^2 |f_t|^{p+2}, \int_{0}^{\infty} \int_{M} |\nabla f_t|^2 |f_t|^p \leq& C \tag{$A_p$} \label{est 1}\\
\int_{0}^{\infty} \int_{M} e^{2u} |f_t|^{p+2} \leq& C \tag{$B_p$} \label{est 2}\\
\lim_{t \to \infty} \int_{M} e^{2u}|f_t|^{p+2}(t) = 0 \tag{$C_p$} \label{est 3}
\end{align}
where the constant $C$ depends on $p$.
\end{prop}

\begin{proof}
The proof consists of two parts.
First, we show that
\begin{equation}
A_p \Longrightarrow B_{p+1}, \quad B_p \Longrightarrow C_p, \quad B_p  \Longrightarrow A_p.
\end{equation}
Then we show the initial step, where we only need to show $B_0$.

By \Cref{higher est}, clearly $A_p$ implies $B_{p+1}$.
Also, from \eqref{f_t p+2}, $B_p$ implies $C_p$.
To see this, suppose
\[
\limsup_{t \to \infty} \int_{M} e^{2u}|f_t|^{p+2}(t) = m_{sup} > m_{inf} = \liminf_{t \to \infty} \int_{M} e^{2u}|f_t|^{p+2}(t).
\]
Then for any $0<\ep < \frac{m_{sup}-m_{inf}}{2}$, we can choose sequences $\{s_n\},\{t_n\}$ such that $s_n, t_n \to \infty$, $s_n<t_n$, and that
\[
\int_{M} e^{2u}|f_t|^{p+2} (t_n) > m_{sup} - \ep, \quad \int_{M} e^{2u}|f_t|^{p+2} (s_n) < m_{inf} + \ep.
\]
But this contradicts to \eqref{f_t p+2} because
\[
0 < \int_{M} e^{2u}|f_t|^{p+2} (t_n) - \int_{M} e^{2u}|f_t|^{p+2} (s_n) \leq 2a(p+1) \int_{s_n}^{t_n} \int_{M} e^{2u}|f_t|^{p+2} \to 0
\]
as $n \to \infty$ by $B_p$.
So $\lim_{t \to \infty} \int_{M} e^{2u}|f_t|^{p+2}(t)$ exists.
It is easy to show that the limit is zero by $B_p$.

To show $B_p$ implies $A_p$, we take $t_2 \to \infty$ and $t_1 = 0$ in \eqref{f_t p+2 full}.

For the initial step $B_0$, the result directly comes from \Cref{volume}.
\end{proof}

\begin{cor} \label{t_n zero}
Let $(f,u)$ be a smooth solution of \eqref{eq2} on $M \times [0,\infty)$.
For any $p \geq 0$, we have
\begin{align}
\lim_{t_1 \to \infty} \lim_{t_2 \to \infty} \int_{t_1}^{t_2} \int_{M} |df|^2 |f_t|^{p+2}, \int_{t_1}^{t_2} \int_{M} |\nabla f_t|^2 |f_t|^p = &0 \tag{$A_p'$} \label{est 1'}\\
\lim_{t_1 \to \infty} \lim_{t_2 \to \infty} \int_{t_1}^{t_2} \int_{M} e^{2u} |f_t|^{p+2} = &0 \tag{$B_p'$} \label{est 2'}.
\end{align}
\end{cor}

\begin{proof}
As above, $A_p'$ implies $B_{p+1}'$.
From \eqref{f_t p+2 full}, together with $C_p$, we conclude that $B_p'$ implies $A_p'$.
Finally, the initial step $B_0'$ is achieved from \Cref{volume}.
\end{proof}

In fact, by taking a time sequence, we can show a better control than $C_p$.
\begin{lemma}
Let $(f,u)$ be a smooth solution of \eqref{eq2} on $M \times [0,\infty)$.
For any $p \geq 0$, there exists $t_n \to \infty$ depending on $p$ such that
\begin{equation}\label{t_n}
\lim_{n \to \infty}  t_n \int_{M} e^{2u}|f_t|^{p+2} (t_n)= 0. \tag{$C_p'$}
\end{equation}
\end{lemma}

\begin{proof}
Suppose not.
Then there is $C_0>0$ and $t_0$ such that for all $t \geq t_0$,
\[
t \int_{M} e^{2u}|f_t|^{p+2}(t) \geq C_0.
\]
But then, for any $t_0 < t_1 < t_2$, by $B_p$,
\[
C \geq \int_{t_1}^{t_2} \int_{M} e^{2u}|f_t|^{p+2} \geq C_0 \int_{t_1}^{t_2} \frac{1}{t} = C_0 (\log t_2 - \log t_1) \to \infty
\]
if we take $t_2 \to \infty$.
This is a contradiction, hence proves $C_p'$.
\end{proof}

\section{Small energy case - Proof of \Cref{main2}}
\label{sec5}

In this section we consider the case when $E(0) \leq \ep_2$ where $\ep_2>0$ will be determined later.
Since the energy is decreasing, we have $E(t) \leq \ep_2$ for all $t$.

We first state several useful lemmas, most of them are similar to \cite{P23}.
Fix $B_{2r}$ and $\varphi$ be a cut-off function supported on $B_{2r}$ such that $\varphi \equiv 1$ on $B_r$, $0 \leq \varphi \leq 1$ and $|\nabla \varphi| \leq \frac{4}{r}$.

\begin{lemma}\label{e pu}
[Park \cite{P23}, Lemma 21.]
Let $(f,u)$ be a smooth solution of \eqref{eq2} on $M \times [0,\infty)$.
For any $s<t$, $p>2$ and for any $q>0$,
\begin{align}
\int_{B_{2r}} e^{pu} \varphi^q (t) \leq& \int_{B_{2r}} e^{pu} \varphi^q (s) + \frac{2b^2(p-2)}{pa} \int_{s}^{t} \int_{B_{2r}} |df|^p \varphi^q, \label{e pu local} \\
\int_{M} e^{pu}(t) \leq& \int_{M} e^{pu}(s) + \frac{2b^2(p-2)}{pa} \int_{s}^{t} \int_{M} |df|^p. \label{e pu global}
\end{align}
\end{lemma}

\begin{lemma} \label{Sobolev}
Let $f$ be a smooth function on $M$.
Then for any $p > 2$,
\begin{equation}
\left( \int_{B_{2r}} |df|^p  \varphi^{\frac{2p}{3}} \right)^{\frac{3}{2p}} \leq C  \left( \int_{B_{2r}} |df|^2 \right)^{\frac{1}{4}} \left( \int_{B_{2r}} |\nabla^2 (f \varphi) |^{\frac{4p}{p+6}} \right)^{\frac{p+6}{4p}}
\end{equation}
where $C$ only depends on $r$ and $p$, and
\begin{equation}
\left( \int_{M} |df|^p \right)^{\frac{3}{2p}} \leq C \left( \int_{M} |df|^2 \right)^{\frac{1}{4}} \left( \int_{M} |\nabla df|^{\frac{4p}{p+6}} \right)^{\frac{p+6}{4p}}
\end{equation}
where $C$ only depends on $M$ and $p$.
\end{lemma}


The following lemma is a version of Gronwall's inequality.
\begin{lemma} \label{Gronwall}
(Gronwall's inequality)
Let $\xi(t)$ be a nonnegative integrable function on $[t_0,t_1]$ which satisfies
\begin{equation} \label{Gronwall int}
\xi(t) \leq C_1(t) \int_{t_0}^{t} \xi(s)ds + C_2(t)
\end{equation}
for integrable functions $C_1(t), C_2(t)$ with $C_1(t) \geq 0$ for all $t \in [t_0,t_1]$.
Then
\begin{equation} \label{Gronwall int2}
\xi(t) \leq C_2(t) + C_1(t) e^{\int_{t_0}^{t} C_1(s) ds} \int_{t_0}^{t} C_2(s) ds.
\end{equation}
\end{lemma}


Now we show the following $W^{2,2}$ estimate.

\begin{prop} \label{W22}
There exists $\ep_2>0$ such that the following holds.
Let $(f,u)$ be a smooth solution of \eqref{eq2} on $M \times [0,\infty)$ with $E(0)\leq \ep_2$.
Then
\begin{equation}
\lim_{t \to \infty} \int_{M} |\nabla df|^2 (t) = 0.
\end{equation}
\end{prop}

\begin{proof}
From the equation $\Delta f + A(df,df) = e^{2u}f_t$, we get
\[
\int_{M} |\Delta f|^2(t) \leq C_N \int_{M} |df|^4(t) + C \int_{M} e^{4u}|f_t|^2(t).
\]
By interpolation,
\[
\int_{M} |df|^4(t) \leq C \int_{M} |df|^2(t) \int_{M} |\nabla df|^2(t) \leq C E(t) \int_{M} |\nabla df|^2(t).
\]
On the other hand, by \Cref{e pu} with $s<t$,
\[
\begin{split}
\int_{M} e^{4u}|f_t|^2(t) \leq& \left( \int_{M} e^{2u}|f_t|^4(t) \right)^{1/2} \left( \int_{M} e^{6u}(t) \right)^{1/2}\\
\leq& \left( \int_{M} e^{2u}|f_t|^4(t) \right)^{1/2} \left( \int_{M} e^{6u}(s) + C \int_{s}^{t} \int_{M} |df|^6  \right)^{1/2}.
\end{split}
\]
From \Cref{Sobolev} with $p=6$, we have
\[
\int_{M} |df|^6(t) \leq C \left( \int_{M} |df|^2(t) \right) \left( \int_{M} |\nabla df|^2(t) \right)^2.
\]
By $L^2$ estimate with above inequalities, we get
\[
\begin{split}
(1 - C C_N^2 E(t)^2)  \left( \int_{M} |\nabla df|^2(t) \right)^2 &\leq  C \int_{M} e^{2u}|f_t|^4(t) \\
& \cdot \left( \int_{M} e^{6u}(s)  +  E(0) \int_{s}^{t}  \left( \int_{M} |\nabla df|^2 \right)^2  \right).
\end{split}
\]

Now choose $\ep_2$ such that $1 - C C_N^2 \ep_2^2 > \frac{1}{2}$, then for all $s < t$, we can apply \Cref{Gronwall}.
Choose $s=0$ and denote
\[
\begin{split}
X(t) =& \left( \int_{M} |\nabla df|^2(t) \right)^2\\
C_1(t) =& 2C E(0) \int_{M} e^{2u}|f_t|^4\\
C_2(t) =& 2C |M| \int_{M} e^{2u}|f_t|^4.
\end{split}
\]
Then the above inequality becomes
\[
X(t) \leq C_1(t) \int_{0}^{t} X(s)ds + C_2(t)
\]
for all $t \geq 0$, hence by \Cref{Gronwall} and \eqref{est 2} with $p=2$,
\begin{equation} \label{W22 der}
\begin{split}
X(t) \leq& C_2(t) + C_1(t) e^{\int_{0}^{t} C_1(s)ds} \int_{0}^{t} C_2(s)ds\\
=& C(|M|,E(0)) \int_{M} e^{2u}|f_t|^4(t) \to 0
\end{split}
\end{equation}
as $t \to \infty$ by \eqref{est 3} with $p=2$.
This completes the proof.
\end{proof}

Note that from \eqref{W22 der} we also have
\begin{equation} \label{W22 int}
\int_{0}^{t} \left( \int_{M} |\nabla df|^2(t) \right)^2 \leq C(|M|,E(0)) \int_{0}^{t} \int_{M} e^{2u}|f_t|^4 < \infty
\end{equation}
for any $t \geq 0$ by \eqref{est 2} with $p=2$.
By Sobolev embedding, we can therefore obtain
\begin{equation} \label{df p}
\lim_{t \to \infty} \int_{M} |df|^p(t) = 0
\end{equation}
for any $p>1$.

Next, we derive the following $W^{2,3}$ convergence.

\begin{prop} \label{W23}
Under the same assumption of \Cref{W22},
\begin{equation}
\lim_{t \to \infty} \int_{M} |\nabla df|^3 (t) = 0.
\end{equation}
\end{prop}

\begin{proof}
From the equation $\Delta f + A(df,df) = e^{2u}f_t$, we get
\[
\begin{split}
\int_{M} |\Delta f|^3(t) \leq& C_N \int_{M} |df|^6(t) + C \int_{M} e^{6u}|f_t|^3(t)\\
\leq& C_N \int_{M} |df|^6(t) + C \left( \int_{M} e^{2u}|f_t|^4(t) \right)^{3/4} \left( \int_{M} e^{18u}(t) \right)^{1/4}.
\end{split}
\]
By \Cref{e pu},
\[
\int_{M} e^{18u} (t) \leq \int_{M} e^{18u}(0) + C \int_{0}^{t} \int_{M} |df|^{18}.
\]
By \Cref{Sobolev} with $p=18$, we have
\[
\int_{M} |df|^{18}(t) \leq C \left( \int_{M} |df|^2(t) \right)^3 \left( \int_{M} |\nabla df|^3 \right)^4.
\]
By $L^3$ estimate with above inequalities, we get
\[
\begin{split}
\left( \int_{M} |\nabla df|^3 (t) \right)^4 \leq& C \left( \int_{M} |\Delta f|^3(t) \right)^4\\
\leq& C C_N^4 \left( \int_{M} |df|^6(t) \right)^4\\
& + C \left( \int_{M} e^{2u}|f_t|^4(t) \right)^3 \left( |M| + C E(0)^3 \int_{0}^{t} \left( \int_{M} |\nabla df|^3 \right)^4 \right).
\end{split}
\]
As above, we can apply \Cref{Gronwall} with
\[
\begin{split}
X(t) =& \left( \int_{M} |\nabla df|^3(t) \right)^4\\
C_1(t) =& C E(0)^3 \left( \int_{M} e^{2u}|f_t|^4(t) \right)^3 \\
C_2(t) =& C |M| \left( \int_{M} e^{2u}|f_t|^4 (t) \right)^3 + C C_N^4 \left( \int_{M} |df|^6(t) \right)^4.
\end{split}
\]
Then
\[
X(t) \leq C_1(t) \int_{0}^{t} X(s)ds + C_2(t)
\]
so by \Cref{Gronwall},
\[
X(t) \leq C_2(t) + C_1(t) e^{\int_{0}^{t} C_1(s)ds} \int_{0}^{t} C_2(s)ds.
\]
From \eqref{est 3} with $p=2$, $m_1 = \sup_{[0,t]} \int_{M} e^{2u}|f_t|^4 < \infty$ is independent on $t$.
Then
\[
\int_{0}^{t} C_1(s) ds \leq C E(0)^3 m_1^2 \int_{0}^{t} \int_{M} e^{2u}|f_t|^4  < \infty
\]
by \eqref{est 2} with $p=2$.
Also, from \eqref{df p} with $p=6$, $m_2 = \sup_{[0,t]} \int_{M} |df|^6 < \infty$ is independent on $t$.
Then from \Cref{Sobolev} with $p=6$ and \eqref{W22 int}, we have
\[
\int_{0}^{t} \int_{M} |df|^6 \leq C E(0) \int_{0}^{t} \left( \int_{M} |\nabla df|^2 \right)^2 < \infty
\]
and hence
\[
\int_{0}^{t} C_2(s)ds \leq C|M| m_1^2 C(E(0)) + CC_N^4 m_2^3 \int_{0}^{t} \int_{M} |df|^6 < \infty.
\]
Therefore, we get
\[
X(t) \leq C_2(t) + C C_1(t) \to 0
\]
by \eqref{est 3} with $p=2$ and \eqref{df p}.
This competes the proof.
\end{proof}

By Sobolev embedding, we now have
\begin{equation}
\lim_{t \to \infty} \|df\|_{L^\infty(M)}(t) = 0.
\end{equation}


For $u$, first note that there is $C_0>0$ such that $|df|(x,t) \leq C_0$ for all $x \in M, t \in [0,\infty)$.
Then for any $t_1 < t$,
\[
\begin{split}
0 < e^{2u}(t) \leq& e^{-2a(t-t_1)}e^{2u}(t_1) + 2b e^{-2at} \int_{t_1}^{t} e^{2as} C_0^2 ds\\
 =& e^{-2a(t-t_1)} e^{2u}(t_1) + \frac{b}{a} C_0^2 \left( 1 - e^{-2a(t-t_1)} \right).
\end{split}
\]
This shows $\lim_{t \to \infty} e^{2u}(x,t) = 0$ for all $x \in M$, or $u(x,t) \to -\infty$ for all $x \in M$ as $t \to \infty$.
This completes the proof of \Cref{main2}.

\section{Large energy case - Proof of \Cref{main3}}
\label{sec6}

In this section we deal with the case where $E(0) > \ep_2$.
Note that if $E(T) \leq \ep_2$ for some $T >0$, then we can consider $f(T)$ as the initial map and $T$ be the initial time and apply \Cref{main2} to get the convergence.
So, in this section we essentially consider the case $E(t) > \ep_2$ for all $t$.

We consider a small ball $B_{2r}$ where the energy over the ball is less than $\ep_2$ for all $t \geq T$ for some $T$.
Due to the fact that $e^{2u}$ involves time integral, this assumption for all $t \geq T$ is needed.
Then we get the local $L^2$ bound for the tension field, but due to the extra term from cut-off function, this bound is only under the time sequence $t_n \to \infty$ defined on \eqref{t_n} with $p=2$.
By classical bubbling analysis, we can get a strong $W^{1,2}$ convergence of $f(t_n)$.

Let us describe the situation more precisely.
\begin{definition} \label{bubble points}
We say $x_0 \in M$ be a sequential bubble point (at time infinite) if there exists $s_n \to \infty$ such that
\begin{equation}
\lim_{r \to 0} \lim_{n \to \infty} \int_{B_{r}(x_0)} |df|^2(s_n) > \ep_2.
\end{equation}

We say $x_0 \in M$ be a uniform bubble point (at time infinite) if
\begin{equation}
\lim_{r \to 0} \lim_{t \to \infty} \int_{B_{r}(x_0)} |df|^2(t) > \ep_2.
\end{equation}

We say $x_0 \in M$ be a regular point (at time infinite) if
\begin{equation}
\lim_{t \to \infty} \int_{B_{r}(x_0)} |df|^2(t) \leq \ep_2
\end{equation}
for all $r>0$ small enough.
\end{definition}

Now in each ball centered at a regular point, we will get $W^{1,2}$ convergence of $f(t_n)$.
Since the total energy is finite, the number of uniform bubble point is at most finite.
But we do not know how large the set of sequential bubble points is.
For example, we can consider a bubble moving around $N$ and winding forever within a circular trajectory, in which case all points on the trajectory are sequential bubble points.

The following $W^{2,2}$ convergence is similar to \Cref{W22}.

\begin{prop} \label{W22-local}
Let $(f,u)$ be a solution of \eqref{eq2} and $t_n$ be as in \eqref{t_n} with $p=2$.
There exists $\ep_2>0$ such that if
\[
\int_{B_{2r}} |df|^2 (t) \leq \ep_2,
\]
for all $t \geq T$ for some $T$, then for all $n$ large enough,
\begin{equation}
 \int_{B_{2r}} |\nabla^2 (f \varphi)|^2 (t_n) \leq C_5  \ep_2
\end{equation}
for some constant $C_5$ only depends on $r,u(T)$ and $E(0)$.
\end{prop}

\begin{proof}
From the equation $\Delta f + A(df,df) = e^{2u}f_t$, we get
\[
\begin{split}
\int_{B_{2r}} |\Delta (f \varphi)|^2 (t) \leq& C_N \int_{B_{2r}} |df|^4 \varphi^2 (t) +  \int_{B_{2r}} e^{4u}|f_t|^2 \varphi^2 (t) + C(\varphi) \int_{B_{2r}} (f^2 + |df|^2)(t).
\end{split}
\]
Without loss of generality, we assume $\int_{B_{2r}} f = 0$ and hence $\int_{B_{2r}} f^p \leq C \int_{B_{2r}} |df|^p$ for any $p \geq 1$.
By Sobolev embedding, for all $t \geq T$,
\[
\begin{split}
\int_{B_{2r}} |df|^4 \varphi^2 (t) \leq& C \left( \int_{B_{2r}} |\nabla (|df|^2 \varphi)| (t) \right)^2 \\
\leq& C \left( \int_{B_{2r}} |df| |\nabla df| \varphi (t) + \int_{B_{2r}} |df|^2 |\nabla \varphi| (t) \right)^2\\
\leq& C \int_{B_{2r}} |df|^2(t) \left( \int_{B_{2r}} |\nabla^2 (f \varphi) |^2 (t) \right) \\
\leq&  C \ep_2 \left( \int_{B_{2r}} |\nabla^2 (f \varphi) |^2 (t) \right) .
\end{split}
\]
On the other hand, by \Cref{e pu}
\[
\begin{split}
\int_{B_{2r}} e^{4u}|f_t|^2 \varphi^2 (t) \leq& \left( \int_{B_{2r}} e^{2u}|f_t|^4(t) \right)^{\frac{1}{2}} \left( \int_{B_{2r}} e^{6u}(t) \varphi^4 \right)^{\frac{1}{2}}\\
\leq& \left( \int_{M} e^{2u}|f_t|^4(t) \right)^{\frac{1}{2}} \left( \int_{B_{2r}} e^{6u} \varphi^4 (T) + C \int_{T}^{t} \int_{B_{2r}} |df|^6 \varphi^4  \right)^{\frac{1}{2}}.
\end{split}
\]
From \Cref{Sobolev} with $p=6$, we have
\[
\begin{split}
\int_{B_{2r}} |df|^6 \varphi^4 (t) \leq& C \left( \int_{B_{2r}} |df|^2(t) \right) \left( \int_{B_{2r}} |\nabla^2 (f \varphi) |^2 (t) \right)^2\\
\leq& C E(0) \left( \int_{B_{2r}} |\nabla^2 (f \varphi) |^2 (t) \right)^2.
\end{split}
\]
By $L^2$ estimate with above inequalities, we get for all $t \geq T$,
\[
\begin{split}
\left( \int_{B_{2r}} |\nabla^2 (f \varphi)|^2(t) \right)^2 \leq& C \left( \int_{B_{2r}} |\Delta (f \varphi) |^2 (t) \right)^2 \\
\leq& C C_N^2 \ep_2^2 \left( \int_{B_{2r}} |\nabla^2 (f \varphi)|^2(t) \right)^2 + C(\varphi) \ep_2^2 \\
& + C \int_{M} e^{2u}|f_t|^4(t) \left( \int_{B_{2r}} e^{6u} \varphi^4 (T) + E(0) \int_{T}^{t} \left( \int_{B_{2r}} |\nabla^2 (f \varphi)|^2(s) \right)^2  \right).
\end{split}
\]

If $\ep_2$ is small enough so that $1 - C C_N^2 \ep_2^2 > \frac{1}{2}$, then we can apply \Cref{Gronwall}.
Denote
\[
\begin{split}
X(t) =& \left( \int_{B_{2r}} |\nabla^2 (f \varphi)|^2(t) \right)^2\\
C_1(t) =& 2C E(0) \int_{M} e^{2u}|f_t|^4
\end{split}
\]
and $C_2$ be a constant such that
\[
2C \int_{B_{2r}} e^{6u} \varphi^4 (T)  \int_{M} e^{2u}|f_t|^4 + 2C(\varphi) \ep_2^2 \leq C_2 \ep_2^2.
\]
Then the above inequality becomes
\[
X(t) \leq C_1(t) \int_{T}^{t} X(s)ds + C_2 \ep_2^2
\]
hence by \Cref{Gronwall}, \eqref{est 2} and \eqref{t_n} with $p=2$,
\begin{equation}
\begin{split}
X(t_n) \leq& C_2 \ep_2^2 + C_1(t_n) e^{\int_{T}^{t_n} C_1(s)ds} \int_{0}^{t_n} C_2 \ep_2^2 ds\\
=& C_2 \ep_2^2 + t_n C_1(t_n) C_2 \ep_2^2 e^{\int_{0}^{t_n} C_1(s)ds}\\
\leq& 2 C_2 \ep_2^2
\end{split}
\end{equation}
for all $n$ large enough.
This completes the proof.
\end{proof}

By Sobolev embedding, we can also obtain
\begin{equation} \label{df p-local}
\int_{B_{r}} |df|^p(t_n) \leq C \ep_2
\end{equation}
for any $p>1$ and for all $n$ large enough.

Hence, we get the uniform $L^2$ bound for the tension field
\begin{equation} \label{L2 tension}
\|\tau(f)(t_n)\|_{L^2(B_{r})}^2 = \int_{B_{r}} |\Delta f|^2 (t_n) + C_N \int_{B_{r}} |df|^4 (t_n) \leq C \ep_2
\end{equation}
for all $n$ large enough, provided $\int_{B_{2r}} |df|^2(t) \leq \ep_2$ for all $t \geq T$.

Now we are ready to prove \Cref{main3}.

\begin{proof}[Proof of \Cref{main3}]

For fixed $r>0$, cover $M$ by $2r$-balls where those of half radius also cover $M$ and each point in $M$ is within at most $10$ balls.
Let $t_n \to \infty$ be a sequence of time defined in \eqref{t_n} with $p=2$.
For each time $t=t_n$, there are at most $10 E(0)/\ep_2$ balls with energy greater than $\ep_2$.
After passing subsequence, still denoted by $t_n$, we may assume the center points of these balls converge.
Let $\Omega(r)$ be the union of those limit balls and $K(r) = M \setminus \Omega(r)$.
We claim that $K(r)$ only consists of regular points.
Because if there is a sequential bubble point $x \in K(r)$, by our assumption, $x$ is a uniform bubble point which contradicts to the fact that $\int_{B_{2r}(x)} |df|^2(t_n) \leq \ep_2$ for all $n$ large enough.

Now on every ball $B_{2r} \subset K(r)$, we can apply classical bubbling analysis for sequence of approximate harmonic maps with uniformly bounded tension fields in $L^2$ to obtain strong $W^{1,2}$ convergence of $f(t_n)$ whose limit map $f_{\infty} : K(r) \to N$ is harmonic.
Note that since $\int_{B_{2r}} |df|^2(t) \leq \ep_2$, there is no energy concentration on $B_{2r}$.

Next, consider a sequence $r_i \to 0$ and take diagonal subsequence, still denoted by $f(t_n)$.
Then $\cup_i K(r_i) = M \setminus \{x_1, \cdots, x_k\}$ for at most finitely many points $x_1, \cdots, x_k$.
Note that these points are uniform bubble points.
And we can extend the harmonic map $f_{\infty}$ to whole manifold $M$ by removable singularity.
Then for any compact set $K \subset M \setminus \{x_1, \cdots, x_k\}$, $f_n$ converges to $f_{\infty}$ in $W^{1,2}$ on $K$.
This completes the proof of \Cref{main3},
\end{proof}

\section*{Acknowledgments}

The author would like to thank to Thomas Parker for valuable comments and simplifications of the argument.
The author also thanks to the anonymous referees for their careful reading and crucial suggestions to the previous versions of this paper.


\bibliographystyle{abbrv}
\bibliography{bib}

\end{document}